\documentclass[12pt,reqno]{amsart}
\usepackage{amsmath, amssymb, amsfonts, amsthm}
\usepackage{mathtools,xcolor,lineno}
\usepackage{hyperref}
\newtheorem{thm}{Theorem}[section]
\newtheorem{lem}[thm]{Lemma}
\newtheorem{prop}[thm]{Proposition}
\theoremstyle{definition}
\newtheorem{dfn}[thm]{Definition}
\newtheorem{exple}[thm]{Example}

\newtheorem{remark}[thm]{Remark}
\theoremstyle{plain}
\newtheorem{cor}[thm]{Corollary}
\numberwithin{equation}{section}
\usepackage[a4paper, top = 1.2in, bottom = 1.2in, left = 1.2in, right = 1.2in]{geometry}
\numberwithin{equation}{section}

\newcommand{\N}{\mathbb{N}}
\newcommand{\Q}{\mathbb{Q}}

\newcommand{\T}{\mathbb{T}}
\newcommand{\Z}{\mathbb{Z}}

\newcommand{\F}{\mathbb{F}}

\newcommand{\mfa}{\mathfrak{a}}

\newcommand{\m}{\mathfrak{m}}
\newcommand{\n}{\mathfrak{n}}
\newcommand{\p}{\mathfrak{p}}
\newcommand{\mfm}{\mathfrak{m}}
\newcommand{\mfn}{\mathfrak{n}}

\newcommand{\mfp}{\mathfrak{p}}

\newcommand{\GL}{\mathrm{GL}}

\newcommand{\Tr}{\mathrm{Tr}}
\newcommand{\new}{\mathrm{new}}

\newcommand{\ra}{\rightarrow}

\newcommand{\mrm}[1]{\mathrm{#1}}

\def\1{1\!\!1}

\newcommand{\pmat}[4]{ \begin{pmatrix} #1 & #2 \\ #3 & #4 \end{pmatrix}}

\newcommand{\psmat}[4]{\bigl( \begin{smallmatrix} #1 & #2 \\ #3 & #4 \end{smallmatrix} \bigr)}

\title[On mod $\p$ congruences for Drinfeld modular forms of level $\p\m$]{On mod $\p$ congruences for Drinfeld modular forms of level $\p\m$}
\author[T. Dalal]{Tarun Dalal}
\email{ma17resch11005@iith.ac.in}
\address{
Department of Mathematics \\
Indian Institute of Technology Hyderabad\\
Kandi, Sangareddy - 502285\\
INDIA. 
}
\author[N. Kumar]{Narasimha Kumar}
\email{narasimha@math.iith.ac.in}
\address{
Department of Mathematics \\
Indian Institute of Technology Hyderabad\\
Kandi, Sangareddy - 502285\\
INDIA. 
}

\keywords{Drinfeld modular forms, $\Theta$-Operator, Atkin-Lehner involution, Congruences, Hecke operators, Eigenvalue}
\subjclass[2010]{Primary 11F33, 11F52 ; Secondary 11F23, 11G09}
\date{}
\begin{document}
\begin{abstract}
In~\cite{CS04}, Calegari and Stein studied the congruences between classical cusp forms $S_k(\Gamma_0(p))$ of prime level
and made several conjectures about them. In~\cite{AB07} (resp., ~\cite{BP11}) the authors proved one of those conjectures
(resp., their generalizations). In this article, we study the analogous conjecture and its generalizations
for Drinfeld modular forms.
\end{abstract}

\maketitle

\section{Introduction and Statements of the main results}
In~\cite{CS04}, Calegari and Stein studied certain relations between the congruences among classical cusp forms $S_k(\Gamma_0(p))$ of prime level 
and the integral closures of their associated Hecke algebras. They have
made a series of conjectures and established connections between them.
One of these conjectures predicts  a precise formula for the index of $\T$ in its integral closure, 
where $\T$ is the algebra of Hecke operators acting on $S_k(\Gamma_0(p),\Z)$ 
generated over $\bar{\Z}_p$.

When $S_k(\Gamma_0(p))$ contains no oldforms (e.g., when $k = 2, 4, 6, 8, 10$, and $14$), then $U_p= - p^{\frac{k}{2}-1} w_p$, where $w_p$ is the Fricke involution. 
Let $S^+_k(\Gamma_0(p))$ (resp., $S^-_k(\Gamma_0(p))$)
denote the plus (resp., minus) eigenspace of $S_k(\Gamma_0(p))$ with respect to $w_p$, and  let $\T^{+} := \T/(U_p + p^{\frac{k}{2}-1 })$ (resp., $\T^{-} := \T/(U_p - p^{\frac{k}{2}-1 })$ )
be the quotient of the Hecke algebra $\T$. Note that $\T^{+}$ (resp., $\T^{-}$) preserves $S_k^+(\Gamma_0(p))$ (resp., $S_k^-(\Gamma_0(p))$). 
Calegari and Stein  (cf.~\cite[Conjecture $3$]{CS04}) conjectured that $\T^+$ and $\T^-$  are integrally closed.
Equivalently, any congruences among the Hecke eigenforms in $S_k(\Gamma_0(p),\bar{\Z}_p)$
can occur only between plus and minus eigenforms for $w_p$. They (cf.~\cite[Conjecture $4$]{CS04}) also conjectured that the eigenvalues of the Fricke involution on $f \in S_2(\Gamma_0(p))$ and $g \in S_4(\Gamma_0(p))$ have opposite signs if there is a  mod $p$ congruence between $g$ and the derivative of $f$. 
In~\cite{AB07}, Ahlgren and Barcau settled this conjecture affirmatively.  
\begin{thm}\label{Ahlgren Congruence}
Let $p\geq 5$ be a prime. Suppose that $f\in S_2(\Gamma_0(p),\bar{\Z}_p )$ and $g\in S_4(\Gamma_0(p), \bar{\Z}_p)$ are eigenforms 
for all Hecke operators and satisfy $\Theta f \equiv g \pmod \mfp$,
where $\mfp$ is the maximal ideal of $\bar{\Z}_p$.
Then the eigenvalues of $w_p$ for $f$ and $g$ have opposite signs.
\end{thm}

Barcau and Pa\c{s}ol (cf.~\cite[\S4]{BP11}) proved that Theorem~\ref{Ahlgren Congruence} continues to hold for level $pN$ with $p\nmid N$,
under an assumption on the weight filtration of $f$. 

\begin{thm}
\label{BP Congruence}
Let $p\geq 5$ be a prime and $N> 4$ be an integer such that $p\nmid N$, and $\p$ be the maximal ideal of $\overline{\Z}_p$. 
Let $f\in S_2(\Gamma_0(pN), \overline{\Z}_p)$ and $g\in S_4(\Gamma_0(pN), \overline{\Z}_p)$ 
be two newforms and satisfy $\Theta f \equiv g \pmod \p$. If $w(f)= p+1$ then the eigenvalues of $w_p^{(pN)}$ for $f$ and $g$ have opposite signs.
\end{thm}

Our main interest lies in studying the conjectures of Calegari and Stein for Drinfeld modular forms and various connections between them.
The present article is a modest first step in this direction where we generalize the
results of~\cite{AB07} and~\cite{BP11} to Drinfeld modular forms of any weight, any type.

 \subsection{Main Results:}
Let $p$ be an odd prime number and $q=p^r$ for some $r \in \N$. 
Suppose $\F_q$ denote the finite field of order $q$. Set $A :=\F_q[T]$ 
and
$K:=\F_q(T)$.
Let $K_\infty=\F_q((\frac{1}{T}))$ be the completion of $K$ 
with respect to the infinite place $\infty$ (corresponding to $\frac{1}{T}$-adic valuation), and denote by $C$ the completion of an algebraic closure of $K_{\infty}$.

Throughout the article, $\p$ denotes a prime ideal of $A$ generated by a monic irreducible polynomial $\pi :=\pi(T)\in A$	of degree $d$ and $\mfm$ denotes an ideal of $A$ generated by a monic polynomial $m:=m(T)\in A$ such that $(\mfp,\mfm)=1$ (i.e., $\pi \nmid m$).


For an ideal $\mfn$ of $A$, we define
$$\Gamma_0(\mfn) :=\big\{\psmat{a}{b}{c}{d}\in \mathrm{GL}_2(A) : c\in \mfn \big\}$$
to be  a congruence subgroup of $\GL_2(A)$.
Let $M_{k,l}(\Gamma_0(\mfn))$ (resp., $M_{k,l}^1(\Gamma_0(\mfn))$) denote the space of 
Drinfeld modular (resp., cusp) forms of weight $k$, type $l$ for $\Gamma_0(\mfn)$.  
Our first result is the following:
\begin{thm}[Theorem~\ref{Drinfeld_Congruence_Thm_1_Text} in the text]
\label{Drinfeld_Congruence_Thm_1}
Suppose that $f\in M_{k,l}^1(\Gamma_0(\mfp))$ and $g\in M_{k+2,l+1}^1(\Gamma_0(\mfp))$ have 
$\p$-integral $u$-series expansions at $\infty$ with $\Theta f \equiv g \pmod {{\p}}$. 
Assume that $w(\overline{F})= (k-1)(q^d-1)+k$ where $F$ is as in Proposition~\ref{ExistenceofF} corresponding to $kf$,
and $(k,p)=1$.
If $f|W_\mfp = \alpha f$ and $g|W_\mfp= \beta g$
with $\alpha,\beta \in \{ \pm 1 \}$, then $\beta = -\alpha$.
\end{thm}

The Ramanujan's $\Theta$-operator, the weight filtration $w(\overline{F})$ of $F$, and 
the Atkin-Lehner involution $W_\p$ are introduced
in \S\ref{Theta operator}, \S\ref{Level1Filtration}, and \S\ref{AL involution} respectively.
In Proposition~\ref{ExistenceofF}, we establish that for any $f\in M_{k,l}^1(\Gamma_0(\mfp\mfm))$,  there exists a Drinfeld modular
form $F \in M_{(k-1)(q^d-1)+k,l}^1(\Gamma_0(\mfm)) $ such that $F \equiv f \pmod \mfp$. 

In Theorem~\ref{Drinfeld_Congruence_Thm_1}, the condition  $w(\overline{F})= (k-1)(q^d-1)+k$ is automatically satisfied for Drinfeld modular forms of weight $2$, type $1$. More precisely, we prove
\begin{cor}
\label{Drinfeld_Congruence_Cor_1}
Suppose that $f\in M_{2,1}^1(\Gamma_0(\mfp))$ and $g\in M_{4,2}^1(\Gamma_0(\mfp))$ have 
$\p$-integral $u$-series expansions at $\infty$ with $\Theta f \equiv g \pmod {{\p}}$. 
Assume that $f \not \equiv 0\pmod \p$. If $f|W_\mfp = \alpha f$ and $g|W_\mfp= \beta g$
with $\alpha,\beta \in \{ \pm 1 \}$, then $\beta = -\alpha$.
\end{cor}

Similar to the classical case, Theorem~\ref{Drinfeld_Congruence_Thm_1} can be extended to the level $\mfp \mfm$, which is described in the following theorem.

\begin{thm}[Theorem~\ref{Drinfeld_Congruence_Thm_2_Text} in the text]
\label{Drinfeld_Congruence_Thm_2}
Let $\mfm$ be an ideal of $A$ generated by a polynomial in $A$ 
which has a prime factor of degree prime to $q-1$ and $\p \nmid \m$.
Suppose that $f\in M_{k,l}^1(\Gamma_0(\mfp\mfm))$ and $g\in M_{k+2,l+1}^1(\Gamma_0(\mfp\mfm))$ have $\p$-integral $u$-series expansions at $\infty$
with $\Theta f \equiv g \pmod {\p}$. Assume that $w(\overline{F})= (k-1)(q^d-1)+k$, where $F$ is as in Proposition~\ref{ExistenceofF} corresponding to $kf$,
and $(k,p)=1$.
If $f|W_\mfp^{(\mfp\mfm)} = \alpha f$ and $g|W_\mfp^{(\mfp\mfm)}= \beta g$ with $\alpha, \beta \in \{ \pm 1 \}$, then $\beta = -\alpha$.
\end{thm}

The (partial) Atkin-Lehner involution $W_\p^{(\p\m)}$ and the weight filtration $w(\overline{F})$ of $F$
are introduced in \S\ref{AL involution} and \S\ref{Level m filtration}, respectively.
We note that in proving Theorem~\ref{Drinfeld_Congruence_Thm_2}, we will make use of the recent work of Hattori~\cite{Hat20}
for which the conditions on $\m$ are necessary.  


There is a significant difference in our approach to prove Theorem~\ref{Drinfeld_Congruence_Thm_1} and Theorem~\ref{Drinfeld_Congruence_Thm_2}.
 We use the structure of Drinfeld modular 
forms for $\GL_2(A)$ in the proof of Theorem~\ref{Drinfeld_Congruence_Thm_1} (cf. \S \ref{Section_Proof_Thm_1}).
We appeal to  the geometry of modular curves and use the recent work of Hattori (cf.~\cite{Hat20}) to prove Theorem~\ref{Drinfeld_Congruence_Thm_2}
(cf. \S \ref{Section_Proof_Thm_2}).

\subsection{Results for $\p$-new forms:}
The space of $\p$-new forms $M_{k,l}^{1,\p-\new}(\Gamma_0(\p \m))$  for level $\p\m$ was introduced 
by Bandini and Valentino (cf.~\cite[Definition 2.14]{BV20}). 
Now, we state Theorem~\ref{Drinfeld_Congruence_Thm_1} and Theorem~\ref{Drinfeld_Congruence_Thm_2} for $\p$-new forms.
They  are natural generalizations of the results of~\cite{AB07} and~\cite{BP11}.

If $f \in M_{k,l}^{1,\p-\new}(\Gamma_0(\p \m))$, then the relation 
$f| W_\p^{(\p\m)} = - \pi^{1-k/2}(f| U_\p)$ (cf. \cite[Theorem 2.16]{BV20})
implies that $f$ is an eigenvector for the $W_\p^{(\p\m)}$-operator if and only if it is an eigenvector for the $U_\p$-operator. 
Note that the normalization here differs from that of~\cite{BV20}. 
For such a Drinfeld modular form $f$, the eigenvalues of $f$ are $\pm \pi^{k/2-1}$ (resp., $\mp 1$)
with respect to the $U_\p$-operator (resp., the $W_\p^{(\p\m)}$-operator).
In fact, the above relation also implies that the eigenvalues of $f$ with respect 
to the $U_\p$-operator and the $W_\p^{(\p\m)}$-operator have opposite signs.

Now, we rephrase our main results in terms of the $U_\p$-operator.
\begin{cor} 
\label{Drinfeld_Congruence_Thm_3}
Let $\m \subseteq A$ be ideal such that either $\m=(1)$ or as in Theorem~\ref{Drinfeld_Congruence_Thm_2}.
Suppose $f\in M_{k,l}^{1,\p-\new}(\Gamma_0(\p\m))$ and $g\in M_{k+2,l+1}^{1,\p-\new}(\Gamma_0(\mfp\m))$
are two Drinfeld modular forms satisfying the hypothesis of Theorem~\ref{Drinfeld_Congruence_Thm_2}. 
If $f$ and $g$ are eigenforms for the $U_{\p}$-operator, then the eigenvalues of 
the $W^{(\p\m)}_{\p}$-operator on $f$ and $g$ have opposite signs.  
\end{cor}
For $\m=(1)$, we get:
\begin{cor}
\label{Drinfeld_Congruence_Cor_3}
Let $f\in M_{2,1}^{1,\p-\new}(\Gamma_0(\p))$ and $g\in M_{4,2}^{1,\p-\new}(\Gamma_0(\mfp))$ 
be two Drinfeld modular forms satisfying the hypothesis of Corollary~\ref{Drinfeld_Congruence_Cor_1}.
If $f$ and $g$ are eigenforms for the $U_{\p}$-operator, then the eigenvalues of 
the $W_{\p}$-operator on $f$ and $g$ have opposite signs.  
\end{cor}


It is natural to wonder what would happen if one drops the assumption on $w(\overline{F})$ in Theorem~\ref{Drinfeld_Congruence_Thm_1}
 and Theorem~\ref{Drinfeld_Congruence_Thm_2}. 
In \S\ref{Section_Counter_Example}, we will show that these theorems may not continue to hold
if we drop the assumption $w(\overline{F})= (k-1)(q^d-1)+k$.



Finally, we note that the results of~\cite{AB07},~\cite{BP11} were proved only for smaller weights and 
it is unknown whether similar results hold for higher weights.
However, our results are valid for Drinfeld modular forms of any weight, any type.

\subsection{An overview of the article}
The article is organized as follows. In \S~\ref{Basic theory of Drinfeld modular forms section}, we recall some basic theory of Drinfeld modular forms. 
In \S~\ref{Background material for the proofs of the main results section}, we introduce certain operators, study the inter-relations between them, and
state two important propositions. In \S~\ref{Section_Proof_Thm_1}, we give a proof of Theorem~\ref{Drinfeld_Congruence_Thm_1}. 
In \S~\ref{Section_Proof_Thm_2}, we recall some results from~\cite{Hat20},~\cite{HatJNT20} and use them to prove Theorem~\ref{Drinfeld_Congruence_Thm_2}. In the final section, i.e., in \S~\ref{Section_Counter_Example}, we show that the assumption $w(\overline{F})= (k-1)(q^d-1)+k$ in Theorem~\ref{Drinfeld_Congruence_Thm_1} and Theorem~\ref{Drinfeld_Congruence_Thm_2} is necessary.


\section{Basic theory of Drinfeld modular forms}
\label{Basic theory of Drinfeld modular forms section}

The theory of Drinfeld modular forms was studied extensively by Goss, Gekeler, and various other authors
(cf. ~\cite{G80}, ~\cite{G80a},~\cite{Gek88},~\cite{GR96} for more details).
In this section, we recall certain theory of Drinfeld modular forms which are needed to prove our results.

There is an equivalence of categories between the category of Drinfeld modules of rank $r$ over a complete subfield $M$ of $C$ containing $K_\infty$ and the category of $M$-lattices of rank $r$ (cf. ~\cite[Theorem 4.6.9]{Gos96}).  
Let $L=\tilde{\pi}A \subseteq C$ be the $A$-lattice of rank $1$ 
corresponding to the rank $1$ Drinfeld module (which is also called the Carlitz module) 
\begin{equation}
\label{Carlitz Module}
\rho_T=TX+X^q,
\end{equation}
where $\tilde{\pi}\in K_\infty(\sqrt[q-1]{-T})$ is defined up to a $(q-1)$-th root of unity.

The Drinfeld upper half-plane $\Omega= C-K_\infty$
has a rigid analytic structure. 
The group $\GL_2(K_\infty)$ acts on $\Omega$ via fractional linear transformations.
Any $x\in K_\infty^\times$ has the unique expression 
$x= \zeta_x\big(\frac{1}{T} \big)^{v_\infty(x)}u_x,$
where $\zeta_x\in \F_q^\times$, and $v_\infty(u_x-1)\geq 0$ ($v_\infty$ is the valuation at $\infty$).

\begin{dfn}
Suppose $k \in \N$, $l \in \Z/(q-1)\Z$. Let $f:\Omega \longrightarrow C$ 
be a rigid holomorphic function. For any $\gamma=\psmat{a}{b}{c}{d}\in \GL_2(K_{\infty})$,
the slash operator $|_{k,l} \gamma$ on $f$ is defined by
\begin{equation}
\label{slash operator}
f|_{k,l} \gamma := \zeta_{\det\gamma}^l\Big(\frac{\det \gamma}{\zeta_{\det(\gamma)}} \Big)^{k/2}(cz+d)^{-k}f(\gamma z).
\end{equation}
\end{dfn}
Note that the slash operator has the following property.
For $i=1,2$, let $k_i \in \N$, $l_i \in \Z/(q-1)\Z$
and $f_i$ be a rigid holomorphic function on $\Omega$.
For $\gamma \in \GL_2(K_{\infty})$,  by~\eqref{slash operator}, we have 
\begin{equation}
\label{Action splits}
\begin{aligned}
f_1|_{k_1,l_1} \gamma \cdot f_2|_{k_2,l_2} \gamma &= \zeta_{\det\gamma}^{l_1}\big(\frac{\det \gamma}{\zeta_{\det(\gamma)}} \big)^{\frac{k_1}{2}}(cz+d)^{-k_1}f_1(\gamma z) \cdot \zeta_{\det\gamma}^{l_2} \big(\frac{\det \gamma}{\zeta_{\det(\gamma)}} \big)^{\frac{k_2}{2}}(cz+d)^{-k_2} f_2(\gamma z)\\ &= \zeta_{\det\gamma}^{l_1+l_2}\big(\frac{\det \gamma}{\zeta_{\det(\gamma)}} \big)^{\frac{k_1+k_2}{2}}(cz+d)^{-(k_1+k_2)}f(\gamma z)\cdot g(\gamma z) \\ &= (f_1.f_2)|_{k_1+k_2, l_1+l_2} \gamma.
\end{aligned} 
\end{equation}


We now define the Drinfeld modular forms of weight $k$, type $l$ for $\Gamma_0(\mfn)$, as follows:
\begin{dfn}
\label{Definition of DMF}
A rigid holomorphic function $f:\Omega \longrightarrow C$ is said to be a Drinfeld modular form of weight $k$, type $l$ 
for $\Gamma_0(\mfn)$ if 
\begin{enumerate}
\item $f|_{k,l}\gamma= f$ , $\forall \gamma\in \Gamma_0(\mfn)$,
\item $f$ is holomorphic at the cusps of $\Gamma_0(\mfn)$.
\end{enumerate}
The space of Drinfeld modular forms of weight $k$, type $l$ for $\Gamma_0(\mfn)$ is denoted by $M_{k,l}(\Gamma_0(\mfn)).$
Furthermore, if $f$ vanishes at the cusps of $\Gamma_0(\mfn)$, then we say $f$ is a Drinfeld cusp form of
weight $k$, type $l$ for $\Gamma_0(\mfn)$ and the space of such forms is denoted by $M^1_{k,l}(\Gamma_0(\mfn))$.
\end{dfn}

If $k\not \equiv 2l \pmod {q-1}$, then $M_{k,l}(\Gamma_0(\mfn))=\{0\}$. So, without loss of generality, we can assume that $k\equiv 2l \pmod {q-1}$.

Let $u(z) :=  \frac{1}{e_L(\tilde{\pi}z)}$, where $e_L(z):= z{\prod_{\substack{0 \ne \lambda \in L }}}(1-\frac{z}{\lambda})$,
be the exponential function attached to the lattice $L$. 
Then, each Drinfeld modular form  $f\in M_{k,l}(\Gamma_0(\mfn))$ has a unique $u$-series expansion at $\infty$
given by $f=\sum_{i=0}^\infty a_f(i)u^i$.
Since $\psmat{\zeta}{0}{0}{1} \in \Gamma_0(\n)$ for $\zeta \in \F_q^\times$, condition $(1)$ of Definition~\ref{Definition of DMF} 
implies that $a_f(i)=0$ if $i\not\equiv l \pmod {q-1}$.
Hence, the $u$-series expansion of $f$ at $\infty$ can be written as
$$\sum_{0 \leq \ i \equiv l \mod (q-1)} a_f(i)u^{i}.$$ 
Note that any Drinfeld modular form of type $> 0$ is automatically a cusp form. 


\subsection{Examples}
We now give some examples of Drinfeld modular forms. 
\begin{exple}[\cite{G80}]
\label{Normalized Eisenstein series g_d}
Let $d\in \N$. For $z\in \Omega$, the function
\begin{equation*}
g_d(z) := (-1)^{d+1}\tilde{\pi}^{1-q^d}L_d \sum_{\substack{a,b\in \F_q[T] \\ (a,b)\ne (0,0)}} \frac{1}{(az+b)^{q^d-1}},
\end{equation*}
is a Drinfeld modular form of weight $q^d-1$, type $0$ for $\mathrm{GL}_2(A)$,
where $\tilde{\pi}$ is the Carlitz period, and $L_d:=(T^q-T)\ldots(T^{q^d}-T)$ is the least common multiple of all monic polynomials of degree $d$.
We refer to $g_d$ as the normalized Eisenstein series of weight $q^d-1$,  type $0$ for $\GL_2(A)$.
\end{exple}

\begin{exple}[\cite{G80a}]
For $z\in \Omega$, the function
\begin{equation*}
\Delta(z) := (T-T^{q^2})\tilde{\pi}^{1-q^2}E_{q^2-1} + (T^q-T)^q\tilde{\pi}^{1-q^2}(E_{q-1})^{q+1},
\end{equation*}
is a Drinfeld cusp form of weight $q^2-1$,  type $0$ for $\mathrm{GL}_2(A)$,
where $E_k(z)= \sum_{\substack{(0,0)\ne (a,b)\in A^2}}\frac{1}{(az+b)^k}.$
The $u$-series expansion of $\Delta$ at $\infty$ is given by $-u^{q-1}+ \cdots$. 
\end{exple}
\begin{exple}[Poincar\'e series]
\label{Poincare Series}

For $z \in \Omega$, define
\begin{equation*}
 h(z) := \sum_{\gamma\in H\char`\\ \GL_2(A)} \frac{\det \gamma . u(\gamma z)}{(cz+d)^{q+1}},
\end{equation*}
where $H=\big\{\psmat{*}{*}{0}{1}\in \GL_2(A)\big\}$ and $\gamma = \psmat{a}{b}{c}{d}\in \GL_2(A)$.
Then $h$ is a Drinfeld cusp form of weight $q+1$, type $1$ for $\mathrm{GL}_2(A)$ (cf.~\cite{Gek88}).
The $u$-series expansion of $h$ at  $\infty$ is given by $ -u -\cdots$.
\end{exple}



\begin{exple}
In \cite{Gek88}, Gekeler defined the function
\begin{equation*}
 E(z):= \frac{1}{\tilde{\pi}} \sum_{\substack{a\in \F_q[T] \\ a \ \mathrm{monic}}} \bigg( \sum_{b\in \F_q[T]} \frac{a}{az+b} \bigg)
\end{equation*}
which is analogous to the Eisenstein series of weight $2$ over $\Q$. 
The function $E$ is not modular, but it satisfies the following transformation rule
\begin{equation}
\label{Etransformation}
E(\gamma z) = (\mathrm{det} \gamma)^{-1} (cz+d)^2E(z) - c\tilde{\pi}^{-1}(\mathrm{det}\gamma)^{-1}(cz+d)
\end{equation}
for $\gamma = \psmat{a}{b}{c}{d}\in \GL_2(A)$.
%
In the proofs of Theorem~\ref{Drinfeld_Congruence_Thm_1} and Theorem~\ref{Drinfeld_Congruence_Thm_2},
we use the function $E(z)$ extensively.
\end{exple}

\subsection{Congruences and $\Theta$-operator:}
We now define the notion of a congruence between two Drinfeld modular forms.


\begin{dfn}
Let  $f=\sum_{n\geq 0} a_f(n)u^n$ be a formal $u$-series in $K[[u]]$. 
We define $$v_\mfp(f) := \inf_n v_\mfp (a_f(n)),$$ 
where $v_\mfp(a_f(n))$ is the $\mfp$-adic valuation of $a_f(n)$.  
We say $f$ has a $\p$-integral $u$-series expansion if $v_\p(f)\geq 0$.

\end{dfn}

\begin{dfn}[Congruence]
\label{definition for congruence of modular forms}
Let $f= \sum_{n\geq 0}a_f(n)u^n$ and $g= \sum_{n \geq 0}a_g(n)u^n$ be two  $u$-formal  series in $K[[u]]$. 
We say that $f\equiv g \pmod \p$ if $v_\p(f-g) \geq 1$.
\end{dfn}

By~\cite[Corollary 6.12]{Gek88}, we have $g_d\equiv 1 \pmod \mfp$.
A similar congruence holds for the Eisenstein series $E_{p-1}$ in the classical case giving an analogy between $g_d$ and $E_{p-1}$. 
Thus, one would expect that $g_d$  plays an essential role 
in the theory of Drinfeld modular forms.


\subsubsection{$\Theta$-operator:}\label{Theta operator}
For Drinfeld modular forms, there is an analogue of the Ramanujan's $\Theta$-operator, which is defined as
$$\Theta := \frac{1}{\tilde{\pi}}\frac{d}{dz}= -u^2\frac{d}{du}.$$

The $\Theta$-operator does not preserve modularity, but it preserves quasi-modularity. 
However, one can  perturb the $\Theta$-operator to create an operator which preserves modularity.
\begin{dfn}\cite[(8.5)]{Gek88}
For $k \in \N$ and $l \in \Z/(q-1)\Z$, we define the operator
$\partial_k:M_{k,l}(\Gamma_0(\n))\ra M_{k+2,l+1}(\Gamma_0(\n))$ by
\begin{equation}
\label{2nd Derivative}
\partial_k f := \Theta f + kEf. 
\end{equation}
\end{dfn}
For simplicity, we write $\partial$ instead of $\partial_k$ if the weight $k$ is clear from the context.
We conclude this section by recalling the following congruence:

\begin{thm}\cite[Theorem 1.1]{Vin10}
\label{Econgruence}
$E\equiv -\partial_{q^d-1}(g_d) \pmod \mfp.$
\end{thm}

\section{Background material for the proofs of the main results}
\label{Background material for the proofs of the main results section}

We begin by introducing the (partial) Atkin-Lehner involutions, the modified Drinfeld modular form $E^*$ and the trace operators.

\subsection{Atkin-Lehner involutions}
\label{AL involution}

Let $\m=(m)$ and $\n=(n)$ be two ideals of $A$, where $m$ and $n$
are non-constant monic polynomials, such that $m || n$, i.e., $m \mid n$
with $(m,n/m)=1$. The following definition can be found in~\cite[Page 331]{Sch96}.
\begin{dfn}
\label{definition Atkin-Lehner involution}
The (partial) Atkin-Lehner involution $W_\m^{(\n)}$ is defined by the action of  
$\psmat {am}{b}{cn}{dm}$ on $M_{k,l}(\Gamma_0(\n))$, where $a,b, c,d\in A$ 
are such that $adm^2-bcn=\zeta\cdot m$ for some $\zeta\in \F_q^*$.
\end{dfn}

The following proposition shows that the operator $W_\m^{(\n)}$ is well-defined.
\begin{prop}\label{Atkin-Lehner representative}
Let $W_\m^\prime=\psmat{a^\prime m}{b^\prime}{c^\prime n}{d^\prime m}$, and 
$W_\m^{\prime\prime}= \psmat{a^{\prime\prime} m}{b^{\prime\prime}}{c^{\prime\prime} n}{d^{\prime\prime} m}$ 
be two representatives for the Atkin-Lehner involution $W_\m^{(\n)}$. Then,
\begin{equation*}
W_\m^\prime \Gamma_0(\n) = \Gamma_0(\n)W_\m^{\prime\prime}.
\end{equation*}
\end{prop}
\begin{proof}
A straightforward calculation shows that $W_\m^\prime\Gamma_0(\n)W_\m^{{\prime\prime}^{-1}}\subseteq \Gamma_0(\n)$ and \\ 
$W_\m^{\prime^{-1}}\Gamma_0(\n)W_\m^{\prime\prime}\subseteq \Gamma_0(\n).$
Hence the result follows.
\end{proof}

\subsection{Action of Atkin-Lehner Operator:}
Recall that $\p$ denotes a prime ideal of $A$ generated by a monic irreducible polynomial $\pi := \pi(T) \in A$ of degree $d$. 
Henceforth, $\m \subseteq A$ denotes an ideal of $A$ generated by 
a monic polynomial $m := m(T)\in A$ such that $(\p,\m)=1$ (i.e., $\pi \nmid m$).

Since $(\pi, m)=1$, 
we take $W_\p^{(\p\m)} := \psmat{\pi}{b}{\pi m}{d\pi}$ where $b,d\in A$, such that $d\pi^2-b\pi m=\pi$.  An easy verification shows that $W_\p^{(\p\m)}.W_\p^{(\p\m)} = \psmat{\pi}{0}{0}{\pi}\gamma$ for some $\gamma\in \Gamma_0(\p\m)$. 
This shows that $W_\p^{(\p\m)}$ acts as an involution on $M_{k,l}(\Gamma_0(\p\m))$. 
If $f\in M_{k,l}(\Gamma_0(\p\m))$ such that $f|_{k,l}W_\p^{(\p\m)}=\alpha f$ for $\alpha\in C\char`\\  \{0\}$, then 
we must have $\alpha^2=1$, i.e.,  $\alpha \in \{ \pm 1\}$.

For $f \in M_{k,l}(\Gamma_0(\p))$, the actions of $W^{(\p)}_{\mfp}$ and $W^{(\p\m)}_{\mfp}$ on $f$ are the same.
If $\m=(1)$, then we denote $W^{(\p)}_{\mfp}$ by $W_{\mfp}$ for simplicity.
In order to calculate the action of $W_{\mfp}^{(\p\m)}$ on some class of modular forms, we need to define the $U_\mfp$-operator.

\subsection{$U_\mfp$-operator and $V_\p$-operator}
For a rigid analytic function $f : \Omega \longrightarrow C$, 
we define:
\begin{equation*}
f| {U_\mfp}(z)= \frac{1}{\pi} \sum_{\substack{\lambda\in A\\ \mathrm{deg}(\lambda)<d}} f
                \big(\frac{z+\lambda}{\pi}\big),
                  \quad
f|V_\mfp(z)=f(\pi z).
\end{equation*}

\subsection{Construction of $E^*$ and its properties:}
We know that $E$ is not a Drinfeld modular form. The following proposition shows how to construct a Drinfeld modular forms using the function $E$.
\begin{prop}
\label{CongrueneofE}
\label{Atkin0nE*}
The function $E^* (z):= E(z) - \pi E|V_\p (z)$ is a Drinfeld modular form of weight $2$, type $1$ for $\Gamma_0(\p)$. Moreover, we have $E^*|_{2,1}W_\mfp=-E^*.$
\end{prop}
  \begin{proof}
  An easy computation using~\eqref{Etransformation} shows that 
  $E^*(\gamma z) = (\det \gamma)^{-1} (cz+d)^2E^*(z)$
  for any $\gamma=\psmat{a}{b}{c}{d}\in \Gamma_0(\p)$.
  Since $E$ and $E|{V_\p}$ are holomorphic on $\Omega$, the function $E^*$ is also holomorphic on $\Omega$.
  Now, it remains to check the holomorphicity at the cusps of $\Gamma_0(\p)$.
   
By~\cite[Proposition 6.7]{Gek01}, we see that $0$ and $\infty$ are the only cusps of $\Gamma_0(\p)$.
The function $E^*$ is holomorphic at $\infty$ since $E$ and $E|V_\p$ are holomorphic at $\infty$.
A straightforward calculation using~\eqref{Etransformation} shows that $E^*(z)|_{2,1}\psmat{0}{-1}{1}{0}$ has a power series expansion in $u$.
Since the matrix $\psmat{0}{-1}{1}{0}$ takes the cusp $\infty$ to the cusp $0$, we conclude that $E^*$ is holomorphic at the cusp $0$.
 Thus $E^*$ is a Drinfeld modular form of weight $2$, type $1$ for $\Gamma_0(\p)$.
 The last part can be verified easily.
\end{proof}

The following two properties of $E^*$ are of importance to us.
\begin{itemize}
 \item If $f\in M_{k,l}(\Gamma_0(\p\m))$ such that $f|_{k,l}W_\mfp^{(\p\m)}=\alpha f$ with $\alpha \in \{ \pm 1\}$, 
then we have $(E^*f)|_{k+2,l+1}W_\mfp^{(\p\m)} = (-\alpha) E^*f$ (cf.~\eqref{Action splits}). 
So in order to change the sign of the eigenvalue of $W_\mfp^{(\p\m)}$ on $f$,
one can simply  multiply $f$ with $E^{*}$.
\item Since $E(z)$ and $E(\pi z)$ have coefficients in $A$
(cf.~\cite[Proposition 3.3]{Vin14}), we have the following congruence 
\begin{equation}
\label{E* is congruenct to E mod p}
E^* \equiv E \pmod \p.
\end{equation}
\end{itemize}

Next, we describe the action of  $W_{\mfp}^{(\p\m)}$ on $\partial_k f$.

\begin{prop}
\label{action of A-L on partial}
Suppose that $f\in M_{k,l}(\Gamma_0(\mfp\m))$ and $f|_{k,l}W_\mfp^{(\p\m)}= \alpha f$ with $\alpha\in \{\pm 1\}$. Then,
\begin{equation}
(\partial_k f)|_{k+2,l+1}W_\mfp^{(\p\m)} = \alpha(\partial_k f - kE^*f).
\end{equation}
\end{prop}
\begin{proof}
For $z\in \Omega$, we have  
\begin{align*}
&(\partial_kf)|_{k+2,l+1}W_\mfp^{(\p\m)}(z)\\
&= \pi^\frac{k+2}{2} (\pi m z+d\pi)^{-(k+2)}(\partial_k f)\Big(\frac{\pi z+b}{\pi m z+d\pi} \Big) \\ 
&= \pi^\frac{k+2}{2}(\pi m z+d\pi)^{-(k+2)} \Big\{\Theta f\Big(\frac{\pi z+b}{\pi m z+d\pi} \Big)+ kE \Big(\frac{\pi z+b}{\pi m z+d\pi} \Big) f\Big(\frac{\pi z+b}{\pi m z+d\pi} \Big)\Big\} \\
&= (\Theta f(z))|_{k+2,l+1}W_\mfp^{(\p\m)} + kE\Big(\frac{\pi z+b}{\pi m z+d\pi}\Big).\pi^\frac{k+2}{2}(\pi m z +d\pi)^{-(k+2)}f\Big(\frac{\pi z+b}{\pi m z+d\pi} \Big) \\
&= \alpha \Theta(f) + \frac{km\alpha f}{\tilde{\pi}(mz+d)}+ k\Big(\pi^2( m z+d)^2E(\pi z)-\frac{m\pi}{\tilde{\pi}}(m z +d)\Big)\frac{1}{\pi(m z+d)^2}f|_{k,l}W_\p^{(\p\m)}  \\ 
&= \alpha \Theta (f) + k\pi E(\pi z)(\alpha f) \\
&= \alpha \Theta (f) + k\alpha E f - k\alpha E f + k\pi E(\pi z)(\alpha f) = \alpha(\partial_kf - kE^* f).
\end{align*}
Here, we have used the equality 
$(\Theta f(z))|_{k+2,l+1}W_\mfp^{(\p\m)} = \alpha \Theta (f) + \frac{km\alpha f}{\tilde{\pi}(mz+d)}$.
\end{proof}


\subsection{Trace operators}
Now, we discuss the trace operators.
\begin{dfn}
For any $\mfa\mid \n$, we define the trace operator
\begin{equation*}
\Tr_{\frac{\n}{\mfa}}^\n : M_{k,l}(\Gamma_0(\n)) \longrightarrow M_{k,l}(\Gamma_0\Big(\frac{\n}{\mfa}\Big))
\end{equation*}
by
\begin{equation*}
\Tr_{\frac{\n}{\mfa}}^\n(f) = \sum_{\gamma\in \Gamma_0(\n)\char`\\ \Gamma_0(\frac{\n}{\mfa})} f|_{k,l}\gamma.
\end{equation*}
\end{dfn}

We will make use of the following proposition to explicitly compute the action of the trace operator
which can be thought of as a generalization of~\cite[Proposition 3.8]{Vin14} from level $\p$ to
level $\p\m$.

\begin{prop}
\label{pm to m}
Let $\p$, $\m$ be as before.
For any $f\in M_{k,l}(\Gamma_0(\p\m))$, we have 
\begin{equation*}
\Tr_{\m}^{\p\m}(f) = f + \pi^{1-k/2}(f|_{k,l}W_\p^{(\p\m)})| U_\p
\end{equation*}
\end{prop}
\begin{proof}
By definition, we have 
\begin{equation*}
\Tr_\m^{\p\m}(f) = \sum_{\gamma\in \Gamma_0(\p\m)\char`\\ \Gamma_0(\m)} f|_{k,l}\gamma.
\end{equation*}
The set $\{\psmat{1}{j}{m}{m j+1}| j\in A, \mathrm{deg}(j)<d \}$, along with the identity matrix, 
is a complete set of representatives for $\Gamma_0(\p\m) \char`\\ \Gamma_0(\m).$ Using the coset representatives, we obtain
\begin{align*}
\Tr_\m^{\p\m}f &= f + \sum_{j\in A, \mathrm{deg}(j)<d} f|_{k,l}\pmat{1}{j}{m}{m j+1}\\ 
&= f+ \sum_{j\in A, \mathrm{deg}(j)<d} f|_{k,l}\pmat{\pi}{b}{\pi m}{\pi d}\pmat{\frac{1}{\pi}}{\frac{j-b}{\pi}}{0}{1}\\ 
&= f+ \sum_{j\in A, \mathrm{deg}(j)<d}(f|_{k,l} W_\p^{(\p\m)})| \pmat{\frac{1}{\pi}}{\frac{j-b}{\pi}}{0}{1}\\ &= f+\sum_{j\in A, \mathrm{deg}(j)<d}{\frac{1}{\pi^{k/2}}}(f|_{k,l}W_\p^{(\p\m)})(\frac{z+j-b}{\pi}) .
\end{align*}
To complete the proof of Proposition~\ref{pm to m}, we require the following lemma. Its proof 
is similar to that of~\cite[Lemma 5.3]{Vin14}, and hence, we omit the details.
\begin{lem}\label{reciprocalroots}
For a fixed $z\in \Omega$ and $a\in A$, the set $\{u(\frac{z+j-a}{\pi})| j\in A, \mathrm{deg}(j)<d\}$ is exactly 
the set of the reciprocal of the roots of the polynomial $\rho_\pi(x)-\frac{1}{u(z)}\in A((u(z)))[x]$ 
(recall that $\rho$ is the rank one Drinfeld module defined by~\eqref{Carlitz Module}). 

\end{lem}
By Lemma~\ref{reciprocalroots}, for a fixed $z\in \Omega$ and $b\in A$, the sets
$\{u(\frac{z+j}{\pi}) \vert j\in A, \mathrm{deg}(j)<d \}$ and  $\{u(\frac{z+j-b}{\pi}) \vert j\in A, \mathrm{deg}(j)<d \}$
are equal. 
Therefore, we conclude that 
\begin{align*}
\Tr_\m^{\p\m}f&= f+\sum_{j\in A, \mathrm{deg}(j)<d}\frac{1}{\pi^{k/2}}(f|_{k,l}W_\p^{(\p\m)})(\frac{z+j-b}{\pi})\\  &= f+\frac{1}{\pi^{k/2}}\sum_{j\in A, \mathrm{deg}(j)<d}(f|_{k,l}W_\p^{(\p\m)})(\frac{z+j}{\pi})\\ &=  f + \pi^{1-k/2}(f|_{k,l}W_\p^{(\p\m)})| U_\p.
\end{align*} 
\end{proof}

\subsection{Key Propositions:}
We are now ready to state and prove the main results of this section.
\begin{prop}
\label{ExistenceofF}
If $f\in M_{k,l}^1(\Gamma_0(\p\m))$ has $\p$-integral $u$-series expansion at $\infty$  such that $f|_{k,l}W_\p^{(\p\m)}= \alpha f$ with $\alpha\in \{\pm 1\}$, 
then there exists $F\in M_{(k-1)(q^d-1)+k,l}^1(\Gamma_0(\m))$ with $\p$-integral $u$-series expansion at $\infty$  such that $f\equiv F \pmod \p.$
\end{prop}
\begin{proof}
For an integer $k\geq 2$, we define
$$
 g_{(k)} := (g_d-\pi^{(q^d-1)/2}g_d\vert_{q^d-1,0}W_\p)^{k-1},$$
where $g_d$ is the Eisenstein series 
of weight $q^d-1$, type $0$ for $\mathrm{GL}_2(A)$ (cf. Example~\ref{Normalized Eisenstein series g_d}). 
Then $g_{(k)}\in M_{(k-1)(q^d-1), 0}(\Gamma_0(\p))$ and it satisfies the following congruences 
\begin{equation}
\label{g_k congruence 1}
g_{(k)} \equiv 1 \pmod \p 
\end{equation}
and
\begin{equation}
\label{g_k congruence 2}
g_{(k)}|_{(k-1)(q^d-1),0}W_\mfp \equiv 0 \pmod {\mfp^{\frac{(k-1)(q^d-1)}{2}+k-1}},
\end{equation}
(cf. \cite[Page 32]{Vin14} for more details).

Since  $f\in M_{k,l}^1(\Gamma_0(\p\m))$ has $\p$-integral $u$-series expansion at $\infty$, we have $v_\p(f)\geq 0$. The function
 $fg_{(k)}$ is a Drinfeld cusp form of weight $(k-1)(q^d-1)+k$, type $l$ for $\Gamma_0(\p\m)$
with $\p$-integral $u$-series expansion at $\infty$. 
Thus, $\Tr_\m^{\p\m}fg_{(k)}$ is a Drinfeld cusp form of weight $(k-1)(q^d-1)+k$, type $l$ for $\Gamma_0(\m)$.

From Proposition \ref{pm to m}, we obtain
\begin{equation*}
\Tr_{\m}^{\p\m}(fg_{(k)}) - fg_{(k)}	 = \pi^{1-\frac{k+(k-1)(q^d-1)}{2}}(fg_{(k)}|_{(k-1)(q^d-1)+k,l}W_\p^{(\p\m)})| U_\p.
\end{equation*}
Since $v_\p(f| U_\p)\geq v_\p(f)$ (cf.~\cite[Corollary 3.2]{Vin14}), it follows that:
\begin{align*}
 v_\p(\Tr_{\m}^{\p\m}(&fg_{(k)}) - fg_{(k)}) \\
 & \geq 1- \frac{(k-1)(q^d-1)+k}{2} + v_\p(fg_{(k)}|_{(k-1)(q^d-1)+k,l}W_\p^{(\p\m)})\\ 
              &= 1- \frac{(k-1)(q^d-1)+k}{2}+v_\p(f|_{k,l}W_\p^{(\p\m)})+ v_\p(g_{(k)}|_{(k-1)(q^d-1),0} W_\p) \\ 
 &\underset{\eqref{g_k congruence 2}}{=} 1-\frac{(k-1)(q^d-1)+k}{2}+ v_\p(f|_{k,l}W_\p^{(\p\m)}) +\frac{(k-1)(q^d-1)}{2} +k-1 \\
 &= \frac{k}{2} + v_\p(f|_{k,l}W_\p^{(\p\m)})\\ 
 & = \frac{k}{2} + v_\p(f) \geq \frac{k}{2} \geq 1 \qquad (\mrm{since}\ f|_{k,l}W_\p^{(\p\m)}= \alpha f\ \mrm{and} \ v_\p(f)\geq 0 ).
 \end{align*}
 We thus get 
\begin{equation}
\label{congruence of trace of f.g_k}
\Tr_\m^{\p\m}fg_{(k)} \equiv fg_{(k)}  \pmod \p.
\end{equation} 
Combining~\eqref{g_k congruence 1} with~\eqref{congruence of trace of f.g_k}, we conclude that
\begin{equation}
\label{final_congruence}
\Tr_\m^{\p\m}fg_{(k)} \equiv f  \pmod \p.
\end{equation}
Thus, the Drinfeld modular form $F:= \Tr_\m^{\p\m}fg_{(k)}  \in M_{(k-1)(q^d-1)+k,l}^1(\Gamma_0(\m))$ 
has  $\p$-integral $u$-series expansion at $\infty$ and it satisfies the conclusion of the proposition.

\end{proof}
\begin{prop}
\label{exists HH}
 Suppose that $h\in M_{k+2,l+1}^1(\Gamma_0(\p\m))$ has a $\p$-integral $u$-series expansion at $\infty$ and $\alpha \in \{ \pm 1 \}$. Then there exists 
 $H\in M_{(k-1)q^d+3,l+1}^1(\Gamma_0(\m))$
 such that $H\equiv \alpha \pi h|_{k+2,l+1}W_\p^{(\p\m)} \pmod \p$.
\end{prop}
\begin{proof}
Since $\alpha \pi h|_{k+2,l+1}W_\mfp^{(\p\m)}.g_{(k)} \in M_{k+2,l+1}^1(\Gamma_0(\p\m))$, by the definition of the trace operator, we get that
$\Tr_{\m}^{(\p\m)}(\alpha \pi h|_{k+2,l+1}W_\mfp^{(\p\m)}.g_{(k)}) \in M_{(k-1)q^d+3,l+1}^1(\Gamma_0(\m)).$

%
By Proposition~\ref{pm to m} and~\eqref{Action splits}, we obtain
\begin{align*}
& v_\p(\Tr_{\m}^{(\p\m)}(\alpha \pi h|_{k+2,l+1}W_\mfp^{(\p\m)}.g_{(k)}) - \alpha \pi h|_{k+2,l+1}W_\mfp^{(\p\m)}.g_{(k)}) \\
        &= v_\mfp(\pi^{1-\frac{(k-1)q^d+3}{2}}\alpha \pi(((h|_{k+2,l+1}W_\mfp^{(\p\m)}).g_{(k)})|_{\color{black}{(k-1)(q^d-1)+k+2,l+1}}W_\mfp^{(\p\m)})|U_\mfp) \\ 
        &= v_\mfp(\alpha \pi^{2-\frac{(k-1)q^d+3}{2}} (((h|_{k+2,l+1}W_\mfp^{(\p\m)})|_{k+2,l+1}W_\mfp^{(\p\m)}).{\color{black}{(g_{(k)}|_{(k-1)(q^d-1),l+1}}}W_\mfp ))|U_\mfp)  \\ 
        &= v_\mfp(\alpha \pi^{2-\frac{(k-1)q^d+3}{2}}(h.(g_{(k)}|_{(k-1)(q^d-1),0}W_\mfp))|U_\mfp) \\ 
        & \geq v_\mfp(\alpha \pi^{2-\frac{(k-1)q^d+3}{2}}g_{(k)}|_{(k-1)(q^d-1),0}W_\mfp)\qquad (\mrm{since} \ v_\mfp(f|U_\mfp) \geq v_\mfp(f) \ \mrm{and} \ v_p(h) \geq 0) \\
        & \geq \frac{(k-1)(q^d-1)}{2}+k-1+2-\frac{(k-1)q^d+3}{2} =\frac{k}{2} \geq 1\qquad (\mrm{using}~\eqref{g_k congruence 2}).
\end{align*}
We thus get
\begin{equation}
\label{congruence of trace of h A-L.g_k}
\Tr_{\m}^{(\p\m)}(\alpha \pi h|_{k+2,l+1}W_\mfp^{(\p\m)}.g_{(k)}) \equiv \alpha \pi h|_{k+2,l+1}W_\mfp^{(\p\m)}.g_{(k)} \pmod \p.
\end{equation}
Combining~\eqref{g_k congruence 1} with~\eqref{congruence of trace of h A-L.g_k}, we conclude that
$$\Tr_{\m}^{(\p\m)}(\alpha \pi h|_{k+2,l+1}W_\mfp^{(\p\m)}.g_{(k)}) \equiv \alpha \pi h|_{k+2,l+1}W_\mfp^{(\p\m)} \pmod \p.$$
Thus, the Drinfeld modular form $H:=\Tr_{\m}^{(\p\m)}(\alpha \pi h|_{k+2,l+1}W_\mfp^{(\p\m)}.g_{(k)}) \in M_{(k-1)q^d+3,l+1}^1(\Gamma_0(\m))$ satisfies the conclusion of the proposition.

\end{proof}

\begin{remark}  
The above result is true for any $\alpha\in K$ with $v_\p(\alpha)\geq 0$. 
Throughout the article, we work with $\alpha \in \{ \pm 1 \}$, so 
we restrict ourselves in deviating from it.
\end{remark}
Now, we are ready to prove Theorem~\ref{Drinfeld_Congruence_Thm_1} and Theorem~\ref{Drinfeld_Congruence_Thm_2}

\section{Proof of Theorem~\ref{Drinfeld_Congruence_Thm_1}}
\label{Section_Proof_Thm_1}
Before going into the proof of Theorem~\ref{Drinfeld_Congruence_Thm_1}, we recall the notion of weight filtration for Drinfeld modular forms 
for $\GL_2(A)$ and list some of its properties.
Let $M_k$ denote the space of Drinfeld modular forms of weight $k$ (any type) for $\mathrm{GL}_2(A)$.

\subsection{Filtration for level $1$ case}
\label{Level1Filtration}
Recall that $\mfp$ denotes a prime ideal of $A$ generated by a monic irreducible polynomial $\pi :=\pi(T)$ of degree $d$.  
Let $f$ be a Drinfeld modular form of weight $k$, type $l$ for $\GL_2(A)$ with $\p$-integral $u$-series expansion at $\infty$.

\begin{dfn}
If $f \not \equiv 0 \pmod \p$, then we define the weight filtration $w(\overline{f})$ of $f$ as 
$$w(\overline{f}) :=\inf \{k_0| \exists f^\prime\in M_{k_0} \ \mathrm{with} \ {f} \equiv {f^\prime} \pmod {\p} \}.$$
If $f\equiv 0 \pmod \p$, then we define $w(\overline{f})=-\infty$.
Since the weight filtration of $f$ is defined mod $\p$, we choose to write $w(\overline{f})$ rather than $w(f)$.
\end{dfn}

To discuss some properties of $w(\overline{f})$,
we recall the structure of the ring $M(\GL_2(A))=\bigoplus_{k,l} M_{k,l}(\GL_2(A))$.
By~\cite[Theorem 5.13]{Gek88}, we have $M(\GL_2(A))=C[g_1,h].$ In particular, every Drinfeld modular form corresponds 
to a unique isobaric polynomial in $g_1$ and $h$ over $C$. 
Let $A_d(X,Y)$ and $B_d(X,Y)$ be the isobaric polynomials attached to $g_d$ and $\partial(g_d)$ respectively,
i.e., $A_d(g_1,h) = g_d$ and $B_d(g_1,h)=\partial(g_d)$.

In~\cite{Gek88}, ~\cite{Vin10}, the authors proved the following properties of $w(\overline{f})$.
\begin{thm}
\label{Drinfeld level 1 Filtration}
Let $f\in M_{k,l}(\mathrm{GL}_2(A))$ and $f= \phi(g_1,h)$ where $\phi(X,Y)$ is the isobaric polynomial attached to $f$.
Then the following hold.
\begin{enumerate}
\item If $f \not \equiv 0 \pmod \p$, then $w(\overline{f})\equiv k \pmod {q^d-1}$,
\item $w(\overline{f})<k$ if and only if $\overline{A}_d| \overline{\phi}$, where $\overline{U}$ denotes the reduction of $U \pmod \p$.
\item $\overline{A_d}(X,Y)$ shares no common factor with $\overline{B_d}(X,Y)$.
\end{enumerate}
\end{thm}

\subsection{Proof of Theorem~\ref{Drinfeld_Congruence_Thm_1}}
Let us recall the statement of this theorem for the convenience of the reader.
\begin{thm}
\label{Drinfeld_Congruence_Thm_1_Text}
Suppose that $f\in M_{k,l}^1(\Gamma_0(\mfp))$ and $g\in M_{k+2,l+1}^1(\Gamma_0(\mfp))$ have 
$\p$-integral $u$-series expansions at $\infty$ with $\Theta f \equiv g \pmod {{\p}}$. 
Assume that $w(\overline{F})= (k-1)(q^d-1)+k$ where $F$ is as in Proposition~\ref{ExistenceofF} corresponding to $kf$,
and $(k,p)=1$.
If $f|W_\mfp = \alpha f$ and $g|W_\mfp= \beta g$
with $\alpha,\beta \in \{ \pm 1 \}$, then $\beta = -\alpha$.
\end{thm}
\begin{proof}
We shall prove this theorem   by contradiction.  
Suppose that $\beta = \alpha$. 

Since $E^* \equiv E \pmod \p$ (cf. \eqref{E* is congruenct to E mod p}) and $\Theta f \equiv g \pmod {{\p}}$, we have  
$$\partial f \equiv g + kE^*f \pmod {\p} \qquad (\mrm{cf}.~\eqref{2nd Derivative}).$$
Hence, there exists $h\in M_{k+2,l+1}^1(\Gamma_0(\p))$ with $v_\p(h)\geq 0$ such that 
\begin{equation}
\label{Bh exists}
 g- \partial f + kE^*f = \pi h  .
\end{equation}
Applying $W_\p$ on both sides of~\eqref{Bh exists}, we obtain
\begin{equation}
\alpha(g-\partial f) = \pi h|_{k+2,l+1}W_\p \qquad (\mrm{cf}.\ \mrm{Proposition}~\ref{action of A-L on partial}).
\end{equation}
Combining this with~\eqref{Bh exists}, we get
\begin{equation} 
 kE^*f = \pi h - \alpha \pi h|_{k+2,l+1}W_\p 
\end{equation}
and hence
\begin{equation*} 
 kE^*f \equiv - \alpha \pi h|_{k+2,l+1}W_\p \pmod \p.
\end{equation*}
Since $v_\p(E^*)\geq 0$ and $v_\p(f) \geq 0$, setting $\m=(1)$ in Proposition~\ref{ExistenceofF} , there exists $F\in M_{(k-1)q^d+1,l}^1(\GL_2(A))$ such that $kf \equiv F \pmod \p.$ Hence 
\begin{equation}
\label{BnewcongE}
 E^*F  \equiv  - \alpha \pi h|_{k+2,l+1}W_\p \pmod \p.
\end{equation}
By~\eqref{BnewcongE} and Proposition~\ref{exists HH} with $\m=(1)$, we obtain
$$E^*F\equiv -H \pmod \p,$$ where $H\in M_{(k-1)q^d+3,l+1}^1(\GL_2(A))$ such that $H\equiv \alpha \pi h|_{k+2,l+1}W_\p \pmod \p$.

Recalling that $E^*\equiv E \equiv -\partial(g_d) \pmod \mfp$, we get 
\begin{equation}
\label{B-final congruence}
H\equiv \partial(g_d)F \pmod \p.
\end{equation}
The last congruence implies that $H$ has $\p$-integral $u$-series expansion at $\infty$.
Since  both sides of~\eqref{B-final congruence} are congruent mod $\p$, we have
\begin{equation}
\label{weight filtrations are equal level p case}
w(\overline{H}) = w(\overline{\partial(g_d)F}).
\end{equation}
We now calculate the weight filtration $w(\overline{\partial(g_d)F})$.




Let $\phi$ be the unique isobaric polynomial attached to $F$. Consequently, the unique isobaric polynomial attached to $\partial(g_d)F$ is $B_d\phi$. The weights of $F$ and $\partial(g_d)F$ are $(k-1)(q^d-1)+k$ and $kq^d+2$, respectively. 
The assumption $w(\overline{F})=(k-1)(q^d-1)+k$ implies that $\overline{A_d}\nmid \overline{\phi}$ (cf. Theorem~\ref{Drinfeld level 1 Filtration}(2)). Combining this with Theorem~\ref{Drinfeld level 1 Filtration}(3), we obtain  
\begin{equation}
\label{A_d does not divide B_d.phi}
\overline{A_d}\nmid \overline{B_d} \overline{\phi}.
\end{equation}
Finally,~\eqref{A_d does not divide B_d.phi} and Theorem~\ref{Drinfeld level 1 Filtration}(2) together yield $w(\overline{\partial(g_d)F}) = kq^d+2$.

Since the weight of $H$ is $(k-1)q^d+3$,
we conclude that $$w(\overline{H})\leq (k-1)q^d+3 < kq^d+2 = w(\overline{\partial(g_d)F}),$$
which contradicts~\eqref{weight filtrations are equal level p case}.
Therefore, we must have $\beta=-\alpha$.
\end{proof}

\subsection{Proof of Corollary~\ref{Drinfeld_Congruence_Cor_1}}
Now, arguing as in the proof of Theorem~\ref{Drinfeld_Congruence_Thm_1}, we get $F\in M_{q^d+1,1}^1(\GL_2(A))$ and $H\in M_{q^d+3,2}^1(\GL_2(A))$ when $k=2$ and $l=1$. Since  $f \not \equiv 0 \pmod \mfp$, the congruence
$w(\overline{F})\equiv q^d+1 \pmod {q^d-1}$ implies that the possible values of $w(\overline{F})$ are  $2$ or $q^d+1$.

The space $M(\GL_2(A))$ is generated by $g_1$ and $h$, where $g_1$ is of weight $q-1$
and $h$ is of weight $q+1$. For $q>3$, the weight of $g_1$ is $q-1 > 2$. Therefore, there is no modular form of weight $2$. For $q=3$, we have  $M_{2,l}(\GL_2(A))=\{0\}$ whenever $l \not \equiv 0 \pmod 2$ and $M^1_{2,0}(\mathrm{GL}_2(A))= \{0\}$. 
Thus, $w(\overline{F})$
cannot be $2$. Hence we obtain $$w(\overline{F})=q^d+1.$$ 
Now, the desired result follows from Theorem~\ref{Drinfeld_Congruence_Thm_1}.

\section{Proof of Theorem~\ref{Drinfeld_Congruence_Thm_2}}
\label{Section_Proof_Thm_2}
Before going into the proof of Theorem~\ref{Drinfeld_Congruence_Thm_2}, let us introduce some notations and recall the relevant results
from ~\cite{Hat20} and~\cite{HatJNT20}.  Using them, we shall prove an important proposition  about the weight filtration.
  
\subsection{Geometry of the Drinfeld modular curves}
Let $\mfm = (m)$ be as in Theorem~\ref{Drinfeld_Congruence_Thm_2}, where $m\in A$ is a non-constant monic polynomial.
The conditions on $\m$ allow us to choose a subgroup $\Delta \subseteq (A/\m)^\times$ such that 
the natural inclusion $\F_q^\times\xhookrightarrow{}(A/\m)^\times$ gives $\Delta \oplus \F_q^\times = (A/\m)^\times$. 

The fine moduli scheme $Y_1^\Delta(\m)$ classifies the tuples $(E,\lambda, [\mu])$,
where $E$ is a Drinfeld module of rank $2$ over an $A[1/m]$-scheme $S$, $\lambda$ is a $\Gamma_1(\m)$-structure on $E$, and 
$[\mu]$ is a $\Delta$-structure on $E$  (cf. ~\cite[Page 20]{Hat20} for more details).

Let $E_{\mathrm{un}}^\Delta$ be the universal Drinfeld module over $Y_1^\Delta(\m)$ 
and $\omega_\mathrm{un}^\Delta$ be the sheaf of invariant differential forms on $E_{\mathrm{un}}^\Delta.$
Let $X_1^\Delta(\m)$ be the compactification of $Y_1^\Delta(\m)$. 
Suppose that $R_0$ is a flat $A[1/m]$-algebra which is an excellent regular domain. The invertible sheaf $\omega_{\mathrm{un}}^\Delta$ on $Y_1^\Delta(\m)_{R_0}$
extends to an invertible sheaf $\overline{\omega}_{\mathrm{un}}^\Delta$ on $X_1^\Delta(\m)_{R_0}$ (cf.~\cite[Theorem 5.3]{HatJNT20}). 

Following~\cite[Page 26]{Hat20}, we define
$\Gamma_1^\Delta(\m):= \{\gamma \in \mathrm{SL}_2(A) | \gamma \equiv \psmat{1}{*}{0}{1} \pmod \m\}.$

\begin{dfn}\cite[Definition 4.7]{Hat20}
Let $k$ be an integer and $M$ be an $A[1/m]$-module. The space of Drinfeld modular forms of level $\Gamma_1^\Delta(\m)$ and weight $k$
with coefficients in $M$ is defined by 
$$M_k(\Gamma_1^\Delta(\m))_M := H^0(X_1^\Delta(\m)_{A[1/\m]},(\overline{\omega}_{\mathrm{un}}^\Delta)^{\otimes k} \otimes_{A[1/\m]} M).$$
\end{dfn}

Consider the map $x_{\infty}^{\Delta} : \mrm{Spec}(A[1/m][[x]]) \ra X_1^\Delta(\m)_{A[1/\m]}$ as in~\cite[Theorem 5.3]{HatJNT20}. 

\begin{dfn}[$x$-expansion]
\label{Identification}
For any $f \in M_k(\Gamma_1^\Delta(\m))_M$, we define the $x$-expansion of $f$ at the $\infty$-cusp as the unique power series $f_{\infty}(x) \in A[1/m][[x]] \otimes_{A[1/m]} M$, satisfying
$  (x^\Delta_{\infty})^* (f)  = f_{\infty}(x) (dX)^{\otimes k}.$
\end{dfn}
By~\cite[Proposition 4.8(ii)]{Hat20},  if $f_{\infty}(x) = 0$, then $f=0$, which we refer to as the $x$-expansion principle. 
\begin{remark}
	The $x$-expansion principle implies that  we can consider a modular form  not only as a global section
	but also in terms of its $x$-expansion.
\end{remark}

This definition of Drinfeld modular forms is compatible with the classical Drinfeld modular forms over $C$, 
which are described in~\cite{Gek86},~\cite{Gek88}. In fact, one can show that the $x$-expansion $f_{\infty}(x)$ of $f$ at the $\infty$-cusp agrees with the $u$-series expansion at $\infty$ of the associated classical Drinfeld modular form to $f$ (cf. ~\cite[Page 26]{Hat20} and references therein for more details).


By~\cite[Proposition 4.8(ii)]{Hat20}, for any $k\geq 2$, and any $A[1/m]$-module $M$, we have an isomorphism 
\begin{equation}
\label{base cahnge}
M_k(\Gamma_1^\Delta(\m))_{A[1/m]} \otimes_{A[1/m]} M \cong M_k(\Gamma_1^\Delta(\m))_M.
\end{equation}
Let $A_\p$ be the completion of $A$ at $\p$. By~\eqref{base cahnge}, we obtain an isomorphism
$$M_k(\Gamma_1^\Delta(\m))_{A[1/m]} \otimes_{A[1/m]} A_\p \cong M_k(\Gamma_1^\Delta(\m))_{A_\p},$$
tensoring with $A/\p$, we obtain the following isomorphism
\begin{equation}
\label{Isomorphism}
\begin{aligned}
M_k(\Gamma_1^\Delta(\m))_{A_\p} \otimes_{A[1/m]} A/\p &\cong
M_k(\Gamma_1^\Delta(\m))_{A[1/m]} \otimes_{A[1/m]} (A_\p \otimes_{A[1/m]} A/\p) \\ 
&\cong  M_k(\Gamma_1^\Delta(\m))_{A/\p}.
\end{aligned}
\end{equation}
Let $\widetilde{f}$ denote the image of $f \in M_k(\Gamma_1^\Delta(\m))_{A_\p}$ under the isomorphism~\eqref{Isomorphism}.
By \cite[Corollary 9.4 of Chapter III]{Har77}, the element $\widetilde{f}$ can also be treated as an element of 
$H^0(X_1^\Delta(\m)_{A/\p},(\overline{\omega}_{\mathrm{un}}^\Delta|_{A/\p})^{\otimes k})$.

\begin{remark}
\label{x-expansion and u-expansion of reduction}
For  $f \in M_k(\Gamma_1^\Delta(\m))_{A_\p}$, the $x$-expansion of $\widetilde{f}$ at the $\infty$-cusp is same as the mod $\p$-reduction of the $u$-series expansion of $f$ at $\infty$.
\end{remark}

\subsection{Weight filtration:}
We are now in a position to define the weight filtration for any $f \in M_k(\Gamma_1^\Delta(\m))_{A_\p}$.

\begin{dfn}\label{Level m filtration}
If $f\not \equiv 0\pmod \p$, then we define the weight filtration $w(\overline{f})$ of $f$ as 
\begin{equation}
\label{definition of filtration for level m}
w(\overline{f}) :=\inf \{k_0 |\exists \  f^\prime \in M_{k_0}(\Gamma_1^\Delta(\m))_{A_\p} \ \mathrm{with} \ {f} \equiv {f^\prime} \pmod {\p} \}.
\end{equation}
By $f \equiv f^\prime \pmod \p$, we mean that the corresponding $x$-expansions of $f$ and $f^\prime$ at the $\infty$-cusp are congruent modulo $\p$, i.e., $f_{\infty}(x) \equiv f^\prime_{\infty}(x) \pmod \p$. If $f\equiv 0 \pmod \p$, then define $w(\overline{f})=-\infty$.
 
By~\cite[Theorem 4.16]{Hat20},  we have $w(\overline{f}) \equiv k \pmod {q^d-1}$. 

\end{dfn}

In the proof of Theorem~\ref{Drinfeld_Congruence_Thm_2}, the following proposition about the weight filtration of $f$ is useful. We follow the approach of Gross in~\cite[Page 459]{Gro90} to prove the 
proposition. 

\begin{prop}\label{Lower Filtration}
If $f\in M_{k}(\Gamma_1^\Delta(\m))_{A_\p}$, then $w(\overline{f})<k$ if and only if $\widetilde{f}$ vanishes at all supersingular points of $X_1^\Delta(\m)_{A/\p}$.
\end{prop}

\begin{proof}
Suppose $w(\overline{f})=k^\prime<k$. Then there exists $f^\prime\in M_{k^\prime}(\Gamma_1^\Delta(\m))_{A_\p}$ such that $f\equiv f^\prime \pmod \p$.
Let $\widetilde{f}, \widetilde{f^\prime}$ and $\widetilde{g_d}$  be the images 
of $f, f^\prime $ and $g_d$ respectively under the isomorphism~\eqref{Isomorphism}.  
By line $5$ in the proof of Proposition 4.22 in \cite{Hat20} together with the injectivity of $(4.15)$ in \cite{Hat20},
we get that $\widetilde{g_d}$ divides $\widetilde{f}$.
Since $g_d$ is a lift of the Hasse invariant, we conclude that
$\widetilde{f}$ vanishes at all supersingular points of $X_1^\Delta(\m)_{A/\p}$.

Conversely, suppose that $\widetilde{f}$ vanishes at all supersingular points of $X_1^\Delta(\m)_{A/\p}$. 
Since $\widetilde{g_d}$ vanishes at all supersingular points exactly once and remains non-zero elsewhere  
on $X_1^\Delta(\m)_{A/\p}$, 
 $\widetilde{f}/ \widetilde{g_d}$ defines a holomorphic global section in 
$M_{k-(q^d-1)}(\Gamma_1^\Delta(\m))_{A/\p}$. Let 
$f^\prime\in M_{k-(q^d-1)}(\Gamma_1^\Delta(\m))_{A_\p}$ be a lift of $\widetilde{f}/ \widetilde{g_d}$ 
under ~\eqref{Isomorphism}.  
Thus $f\equiv f^\prime \pmod \p$  
since $(g_d)_{\infty}(x) \equiv 1 \pmod \p$.  
This implies $w(\overline{f})<k$.
\end{proof}

\begin{remark} 
Observing that $\Gamma_1^\Delta(\m)\subset \Gamma_0(\m)$, we get $M_{k,l}(\Gamma_0(\m))\subset M_k(\Gamma_1^\Delta(\m))$. Since the order of the determinant group of $\Gamma_1^\Delta(\m)$ is 1, the type does not play any role for Drinfeld modular forms of level $\Gamma_1^\Delta(\m)$. In fact, for a fixed $k$, all $M_{k,l}(\Gamma_1^\Delta(\m))$ are isomorphic 
(cf.~\cite[Page 49]{Boc}).
\end{remark}

\subsection{Proof of Theorem~\ref{Drinfeld_Congruence_Thm_2}:}
Let us recall the statement of this theorem for the convenience of the reader.
\begin{thm}
\label{Drinfeld_Congruence_Thm_2_Text}
Let $\mfm$ be an ideal of $A$ generated by a polynomial in $A$ 
which has a prime factor of degree prime to $q-1$  and $\p \nmid \m$.
Suppose that $f\in M_{k,l}^1(\Gamma_0(\mfp\mfm))$ and $g\in M_{k+2,l+1}^1(\Gamma_0(\mfp\mfm))$ have $\p$-integral $u$-series expansions at $\infty$
with $\Theta f \equiv g \pmod {\p}$. Assume that $w(\overline{F})= (k-1)(q^d-1)+k$, where $F$ is as in Proposition~\ref{ExistenceofF} corresponding to $kf$,
and $(k,p)=1$.
If $f|W_\mfp^{(\mfp\mfm)} = \alpha f$ and $g|W_\mfp^{(\mfp\mfm)}= \beta g$ with $\alpha, \beta \in \{ \pm 1 \}$,  then $\beta = -\alpha$.
\end{thm}

\begin{proof}
We shall prove this theorem by contradiction.  Suppose that $\beta = \alpha$. 
We follow the argument as in the proof of Theorem~\ref{Drinfeld_Congruence_Thm_1_Text}.

Since $E^* \equiv E \pmod \p$ (cf. \eqref{E* is congruenct to E mod p}) and $\Theta f \equiv g \pmod {{\p}}$, we have 
$$\partial f \equiv g + kE^*f \pmod {\p} \qquad (\mrm{cf}.~\eqref{2nd Derivative}).$$
Hence, there exists $h\in M_{k+2,l+1}^1(\Gamma_0(\p\m))$ with $v_\p(h)\geq 0$, such that 
\begin{equation}
\label{h exists}
 g- \partial f + kE^*f = \pi h  .
\end{equation}
Applying $W_\p^{(\p\m)}$ on both sides, we obtain
\begin{equation}
\alpha(g-\partial f) = \pi h|_{k+2,l+1}W_\p^{(\p\m)} \qquad (\mrm{cf}.\ \mrm{Proposition}~\ref{action of A-L on partial}).
\end{equation} 
Combining this with~\eqref{h exists}, we get
\begin{equation}
kE^*f = \pi h - \alpha \pi h|_{k+2,l+1}W_\p^{(\p\m)}
\end{equation}
and hence
\begin{equation} 
 kE^*f  \equiv - \alpha \pi h|_{k+2,l+1}W_\p^{(\p\m)} \pmod \p.
\end{equation}
Since $v_\p(E^*)\geq 0$ and $v_\p(f)\geq 0$, by Proposition~\ref{ExistenceofF}, there exists $F\in M_{(k-1)q^d+1,l}^1(\Gamma_0(\m))$ such that $kf \equiv F \pmod \p$, hence
\begin{equation}
\label{E*,F and h A-L congrueence level pm}
 E^*F  \equiv  - \alpha \pi h|_{k+2,l+1}W_\p^{(\p\m)} \pmod \p.
\end{equation}
By Proposition~\ref{exists HH} and~\eqref{E*,F and h A-L congrueence level pm}, we obtain 
$$E^*F\equiv -H \pmod \p,$$
where $H \in M_{(k-1)q^d+3,l+1}^1(\Gamma_0(\m))$ such that $H\equiv \alpha \pi h|_{k+2,l+1}W_\p^{(\p\m)} \pmod \p$.
 
 Recalling that $E^*\equiv E \equiv -\partial(g_d) \pmod \mfp$, we  get  
\begin{equation}
\label{final congruence}
H\equiv \partial(g_d)F \pmod \p.
\end{equation}
In particular, the last congruence implies that $H$ has $\p$-integral $u$-series expansion at $\infty$.

Since 
both sides of~\eqref{final congruence} are congruent mod $\p$, we have
\begin{equation}
\label{weight filtrations are equal level pm case}
w(\overline{H}) = w(\overline{\partial(g_d)F}).
\end{equation}
Note that the weight of $\partial(g_d)F$ is $kq^d+2$.
We claim that $w(\overline{\partial(g_d)F})=kq^d+2$.  By Proposition~\ref{Lower Filtration}, it suffices to show that $\widetilde{\partial(g_d)F}$ does not vanish
at least at one supersingular point of $X_1^\Delta(\m)_{A/\p}$. 

The vanishing of the function $\widetilde{g_d}$ at all supersingular points of $X(1)_{A/\p}$ implies that
the function $\widetilde{\partial(g_d)}$ does not vanish at any supersingular point of $X(1)_{A/\p}$ (cf. Theorem~\ref{Drinfeld level 1 Filtration}(3)). Since all the supersingular points of $X_1^\Delta(\m)_{A/\p}$ lie above the supersingular points of $X(1)_{A/\p}$, the function $\widetilde{\partial(g_d)}$ does not vanish
at any supersingular point of  $X_1^\Delta(\m)_{A/\p}$ .
On the other hand, since $w(\overline{F})=(k-1)(q^d-1)+k$, the function $\widetilde{F}$ does not vanish at least at one supersingular point of $X_1^\Delta(\m)_{A/\p}$(cf. Proposition~\ref{Lower Filtration}). Thus $\widetilde{\partial(g_d)}\cdot \widetilde{F}$ does not vanish at least at one supersingular point of $X_1^\Delta(\m)_{A/\p}$, and the claim $w(\overline{\partial(g_d)F})=kq^d+2$ follows.


Since the weight of $H$ is $(k-1)q^d+3$,
we conclude that 
\begin{equation*}
w(\overline{H})\leq (k-1)q^d+3 < kq^d+2=w(\overline{\partial(g_d)F}),
\end{equation*}
which contradicts~\eqref{weight filtrations are equal level pm case}.
Therefore, we must have $\beta=-\alpha$.

\end{proof}

\begin{remark}
In the above argument, we have used the equality   $\widetilde{\partial(g_d)F} = \widetilde{\partial(g_d)}\cdot \widetilde{F}$. This follows from Remark~\ref{x-expansion and u-expansion of reduction} and the equality
$\overline{\partial(g_d)F} = \overline{\partial(g_d)} \cdot \overline{F}$ (where $\overline{X}$ refers to the reduction of the $u$-series expansion of $X$ modulo $\p$). 
\end{remark}

\section{Counterexamples}
\label{Section_Counter_Example}
In this section, we shall show that the assumption  $w(\bar{F})=(k-1)(q^d-1)+k$ is necessary in Theorem~\ref{Drinfeld_Congruence_Thm_1}
and Theorem~\ref{Drinfeld_Congruence_Thm_2}.

\subsection{Eigenforms for $W_\p^{(\p\m)}$:}
\label{Section_Pair_Eigenforms}
Recall that $\p$ denotes a prime ideal of $A$ generated by a monic irreducible polynomial $\pi := \pi(T) \in A$ of degree $d$ and $\m$ denotes an ideal of $A$ generated by 
a monic polynomial $m := m(T)\in A$ such that $(\p,\m)=1$ (i.e., $\pi \nmid m$).
We now discuss the existence of eigenforms for the $W_\p^{(\p\m)}$-operator. 

For $f\in M_{k,l}^1(\Gamma_0(\m))$, we have $f|_{k,l}\psmat{\pi}{0}{0}{1} =  \pi^{k/2} f(\pi z) \in M_{k,l}^1(\Gamma_0(\p\m))$. 
By~\cite[Proposition 3.3]{Vin14}, we get $v_\p(f(\pi z))= v_\p(f|V_\p) \geq v_\p (f)$. This implies that $f|_{k,l}\psmat{\pi}{0}{0}{1} \equiv 0 \pmod \p$ when $f$ has $\p$-integral $u$-series expansion at $\infty$.
\begin{lem}
\label{Exp_Eigenvalues}
If $f\in M_{k,l}^1(\Gamma_0(\m))$, then
\begin{enumerate}
\item $(f+f|_{k,l}\psmat{\pi}{0}{0}{1})|_{k,l}W_\p^{(\p\m)}= f+ f|_{k,l}\psmat{\pi}{0}{0}{1}$,
\item $(f-f|_{k,l}\psmat{\pi}{0}{0}{1})|_{k,l}W_\p^{(\p\m)}= -(f- f|_{k,l}\psmat{\pi}{0}{0}{1})$.
\end{enumerate}
\end{lem}
\begin{proof}
\begin{align*}
(f\pm f|_{k,l}\psmat{\pi}{0}{0}{1})|_{k,l}W_\p^{(\p\m)} &= (f\pm f|_{k,l}\psmat{\pi}{0}{0}{1})|_{k,l}\psmat{\pi}{b}{\pi m}{d\pi} \\ 
                                                     &= f|_{k,l}\psmat{\pi}{b}{\pi m}{d\pi}\pm f|_{k,l}\psmat{\pi}{0}{0}{1}\psmat{\pi}{b}{\pi m}{d\pi} \\ 
                                                     &= f|_{k,l}\psmat{1}{b}{m}{d\pi}\psmat{\pi}{0}{0}{1}\pm f|_{k,l}\psmat{\pi}{b}{m}{d}\psmat{\pi}{0}{0}{\pi}\\ 
                                                     &=  f|_{k,l}\psmat{\pi}{0}{0}{1} \pm f
\end{align*}
\end{proof}
Note that the eigenvectors  $f \pm f|_{k,l}\psmat{\pi}{0}{0}{1}$ are oldforms. 

\subsection{Prototype for a counterexample}
\label{Prototype}
Suppose there exists $f \in M_{k,l}^1(\Gamma_0(\m))$ with $\p$-integral $u$-series expansion at $\infty$
such that $\Theta f \equiv fE \pmod \p$. By definition, we  have
$f \pm f|_{k,l}\psmat{\pi}{0}{0}{1} \in M_{k,l}^1(\Gamma_0(\p\m))$. Clearly,
\begin{equation}
\label{f pm equiv f}
f \pm f|_{k,l}\psmat{\pi}{0}{0}{1} \equiv f \pmod \p.
\end{equation}
The above congruence shows that $w(\bar{F})<(k-1)(q^d-1)+k$, where $F$ is as in 
Proposition~\ref{ExistenceofF}, corresponding to $f \pm f|_{k,l}\psmat{\pi}{0}{0}{1}$. 

By~\eqref{f pm equiv f}, we obtain
\begin{equation}
\Theta(f + f|_{k,l}\psmat{\pi}{0}{0}{1}) \equiv \Theta f \equiv fE \equiv fE^* \equiv (f \mp f|_{k,l}\psmat{\pi}{0}{0}{1})E^* \pmod \p.
\end{equation}

According to Lemma~\ref{Exp_Eigenvalues} and Proposition~\ref{Atkin0nE*}, the modular forms
$f + f|_{k,l}\psmat{\pi}{0}{0}{1}$ and $(f \mp f|_{k,l}\psmat{\pi}{0}{0}{1})E^*$ have the same (resp., opposite) sign under the action of $W_\p^{(\p\m)}$. 
So, the existence of such a function $f$ implies that the assumption on the weight filtration of $F$ is necessary
in Theorem~\ref{Drinfeld_Congruence_Thm_1} and Theorem~\ref{Drinfeld_Congruence_Thm_2}. 


\subsection{Counterexamples:}
Here, we shall produce some Drinfeld modular forms $f$ satisfying $\Theta f \equiv fE \pmod \p$,
so that we can apply the above recipe to produce counterexamples.
\begin{itemize}
 \item An easy computation shows that $\partial_{q^2-1} \Delta=0$, i.e., $\Theta \Delta + (q^2-1) E \Delta=0$. Hence, $\Theta \Delta \equiv  \Delta E \pmod \p$.
Taking $f=\Delta$ in the previous section, we conclude that the assumption  
$w(\overline{F})= (k-1)(q^d-1)+k$ in Theorem~\ref{Drinfeld_Congruence_Thm_1} is necessary.

Note that the weight of $\Delta$ is $q^2-1$ and the type is $0$. Since $q>2$, $q^2-1$ can never be $2$. So, this example
does not contradict Corollary~\ref{Drinfeld_Congruence_Cor_1}.
\item Let $\m$ be as in Theorem~\ref{Drinfeld_Congruence_Thm_2}. Consider any non-zero Drinfeld modular form $g\in M_{k,l}(\Gamma_0(\m))$ with $\p$-integral $u$-series expansion at $\infty$. A straightforward calculation shows that $g^{q^i}\Delta\in M_{kq^i+q^2-1,l}^1(\mathrm{\Gamma_0(\m)})$
and $\partial(g^{q^i}\Delta) = 0$, for $i\geq 1$. 
Taking $f=g^{q^i}\Delta$ in the previous section, we conclude that
the assumption  $w(\overline{F})= (k-1)(q^d-1)+k$ in Theorem~\ref{Drinfeld_Congruence_Thm_2} is necessary.

\end{itemize}
In~\cite{BP11}, the authors have produced an example to demonstrate the necessity of the assumption on the weight filtration in their theorem.
In the function field case, we are able to produce infinitely many examples.

\section*{Acknowledgments}  
The authors are grateful to the anonymous referee for his/her valuable mathematical comments, suggestions  which improved the readability of this article. The authors thank Prof. Dr. Gebhard B\"{o}ckle for informing us about the article~\cite{Hat20} and Prof. Shin Hattori, Prof. C\'{e}cile Armana for helpful discussions. The authors also thank Dr. Pradipto Banerjee for his valuable suggestions.
 
The first author thanks University Grants Commission, India for the financial support provided in the form of Research Fellowship to carry out this research work at IIT Hyderabad. The second author's research was partially supported by the SERB grant MTR/2019/000137.

\bibliographystyle{plain, abbrv}

\end{document}